\documentclass[11pt,a4paper]{article}

\usepackage{amsmath}
\usepackage{amsfonts}
\usepackage{amsthm}
\usepackage[pdftex]{color,graphicx}
\usepackage{graphicx}

\newtheorem{Theorem}{Theorem}

\newtheorem{Proposition}{Proposition}
\newtheorem{Lemma}{Lemma}
\newtheorem{Remark}{Remark}

\newcommand{\ds}{{\mathrm{d}}s}
\newcommand{\dx}{{\mathrm{d}}x}
\newcommand{\dy}{{\mathrm{d}}y}
\newcommand{\Id}{{\mathbf{1}}}
\newcommand{\dom}{{\mathrm{dom}~}}

\newcommand{\var}{\varepsilon}

\begin{document}

\title{Spectrum of the Dirichlet Laplacian in waveguides with parallel cross-sections}
\author{Alessandra A. Verri}

\date{\today}

\maketitle 

\begin{abstract}
Let $\Omega \subset \mathbb R^3$ be a waveguide which is obtained by translating 
a cross-section in a constant direction along an unbounded spatial curve. 
Consider $-\Delta_{\Omega}^D$ the Dirichlet Laplacian operator in $\Omega$.
Under the condition that the tangent vector of the reference curve admits a finite limit at infinity, we find the essential spectrum
of $-\Delta_{\Omega}^D$.
Then, we state sufficient conditions 
that give rise to a non-empty discrete spectrum for $-\Delta_{\Omega}^D$;
in particular, we show that the number
of discrete eigenvalues can be arbitrarily large since the waveguide is thin enough.
\end{abstract}








\section{Introduction}

Let $\Omega$ be an unbounded quantum waveguide in $\mathbb R^n$, $n=2,3$, and denote by $-\Delta_\Omega^D$ the 
Dirichlet Laplacian operator in $\Omega$. The spectrum of this operator has been extensively studied
in the last years \cite{bori1, bori2, bmt, david,briet, popoff01, popoff02, duclosfk, 
clark, pavelduclos, duclos, ekholmk, bookexner, exnerseba, friedcabs, gold, grush01, davidtversusb, davidruled,
davidalde, davidkr, kovsac}. In fact, the subject is non-trivial since the results depend on the geometry of $\Omega$ 
\cite{david, duclosfk, duclos, exnerseba, davidtversusb, davidruled}.
In the particular case where $\Omega$ is a straight waveguide, it is known that its spectrum is purely absolutely continuous and 
there are no
discrete eigenvalues. However, several works show that the situation changes according to the geometry of $\Omega$;
for example, curved waveguides or those  with deformation on its boundary can give rise to a
non-empty discrete spectrum for $-\Delta_\Omega^D$ \cite{brietnew, exnerseba, gold, davidtversusb}.
In the next paragraphs, we recall some papers and results in the literature; after that, we present the goal of this work.

Let $\Gamma: \mathbb R \to \mathbb R^2$ be a $C^\infty$ plane curve parameterized by its arc-length $x$
and denote by $k(x)$ its curvature at the point $\Gamma(x)$. Consider a strip $\Omega$ constructed by moving a bounded 
segment $(a,b) \subset \mathbb R$ along $\Gamma$ with respect to its normal vector field;
$\Omega$ is delimited by two parallel curves.
In the pioneering paper \cite{exnerseba},
the authors studied the spectral problem of $-\Delta_\Omega^D$. In particular, 
they proved the existence of discrete spectrum for the operator under the 
conditions that $k(x)\neq0$, for some $x \in \mathbb R$, and that $k(x)$ decays fast enough at infinity;
in addition, 
the authors had assumed some regularity for $\Gamma$ which were
relaxed latter \cite{gold,davidkr,renger}.

Now, denote by $\Gamma: \mathbb R \to \mathbb R^3$  a $C^3$ spatial curve parameterized by its arc-length $x$
which possesses an appropriate Frenet frame;  $k(x)$ and $\tau(x)$ denote its curvature and torsion at the point $\Gamma(x)$, respectively.
Let $\omega$ be a bounded open connected set in $\mathbb R^2$.
Consider the case where $\Omega$ is an unbounded waveguide obtained 
by moving the cross-section $\omega$ along the curve $\Gamma$ according to the Frenet referential.
At each point of $\Gamma$ the region
$\omega$ also may present a (continuously differentiable) rotation angle $\alpha(x)$. 
In this situation, several aspects of the spectral problem for $-\Delta_\Omega^D$ were studied in  
\cite{briet, popoff01, duclosfk, clark, duclos, ekholmk, bookexner,brietnew, exnerseba, gold, davidtversusb}. However, we emphasize
the paper \cite{davidtversusb} where the authors presented, together with new proofs and results, a review
of the geometric effects of $\Omega$ on  $\sigma(-\Delta_\Omega^D)$.
There was discussed  how the spectrum depends on two independent geometric deformations:
bending ($k \neq 0$) and twisting  ($w$ is not rotationally invariant with respect to the origin and 
$\tau-\alpha' \neq 0$). In particular,
the assumptions $\tau+\alpha'=0$, $k \neq0$ and $k(x) \to 0$, $|x| \to \infty$,
imply $\sigma_{dis}(-\Delta_\Omega^D)\neq \emptyset$.
Some types of deformation in straight waveguides can also give rise to a non-empty discrete spectrum for $-\Delta_\Omega^D$. For example, let $\zeta$ be a bounded function
such that supp $\zeta \subset [-x_0, x_0]$, for some $x_0 > 0$. Consider the case where 
$w$ is a non-circular cross-section and the angle rotation governed by the function $\alpha(x)$ satisfies $\alpha'(x) = \alpha_0 - \zeta(x)$, $\alpha_0 \in \mathbb R$. 
This situation was considered in \cite{briet}. One of the results of the authors is that
the assumption $\int_{-x_0}^{x_0} (\alpha'(x)^2 - \alpha_0^2) \dx < 0$ ensures the existence of discrete eigenvalues for
$-\Delta_\Omega^D$.

In this paper we consider the Dirichlet Laplacian operator restricted to a three-dimensional waveguide whose 
geometry is inspired by a two-dimensional model. Namely, in the recent paper \cite{david} the authors
introduced a new model of strip to study the spectral problem of $-\Delta_\Omega^D$.
In that work,
$\Omega$ is a strip in $\mathbb R^2$ 
which is built by translating a segment oriented in a constant direction along an unbounded curve in the plane. More precisely,
let $h: \mathbb R \to \mathbb R$ be a locally Lipschitz continuous function. Take a positive number $d>0$ and define
$\Omega=\{(x,y) \in \mathbb R^2: h(x) < y < h(x) + d\}$.
Assume that $h(x)$ is differentiable almost everywhere and the derivative $h'(x)$ admits 
a limit at infinity: $h'(x) \to \beta \in \mathbb R \cup \{\infty\}, |x| \to \infty$.
The spectrum of the operator $-\Delta_\Omega^D$ was carefully studied and the model 
covers different effects: purely essential spectrum, discrete spectrum or a combination of both.
In particular, if $h\in L_{loc}^\infty(\mathbb R)$, then $\sigma_{ess} (-\Delta_\Omega^D) = [(1+\beta^2) \pi^2/d^2, \infty)$.
In addition, if $h'(x)^2-\beta^2 \in L^1(\mathbb R)$ and $\int_\mathbb R (h'(x)^2 - \beta^2) \dx < 0$, 
then $\inf \sigma (-\Delta_\Omega^D) <  (1+\beta^2) \pi^2/d^2$, i.e.,
$\sigma_{dis} (-\Delta_\Omega^D) \neq \emptyset$.
In the next paragraphs we present the formal construction of the
waveguide which will be considered in this work  and we give more details of the problem.

Denote by $\{e_1, e_2, e_3\}$ the canonical basis of $\mathbb R^3$. Pick $S \neq \emptyset$;  
an open, bounded and connected subset of~$\mathbb R^2$ with $C^2$-boundary.
Let $f,g: \mathbb R \to \mathbb R$ be differentiable functions so that $f',g' \in L^\infty(\mathbb R)$, and
$r: \mathbb R \to \mathbb R^3$ the spatial curve given by
\begin{equation}\label{refcurveint}
r(x) = (x, f(x), g(x)), \quad x \in \mathbb R.
\end{equation}
Define the mapping 
\begin{equation}\label{cahngeofcoordif1}
\begin{array}{llll}
{\cal L} : &   \mathbb R \times S    & \to      & \mathbb R^3 \\
                       & (x, y_1, y_2)  & \mapsto  & r(x) + y_1 e_2 + y_2 e_3
\end{array}.
\end{equation}
We define the waveguide
\[\Omega := {\cal L}(\mathbb R \times S).\]
Roughly speaking, $\Omega$ is obtained by translating
the region $S$ along the curve $r(x)$ so that, at each point of it,
$S$ is parallel to the plane generated by $\{e_2, e_3\}$.

Let $-\Delta_\Omega^D$ be the Dirichlet Laplacian operator in $\Omega$. 
In other words,
$-\Delta_{\Omega}^D$ is the self-adjoint operator associated with the quadratic form
\begin{equation}\label{quadinform1}
a(\varphi) : = 
\int_{\Omega}  |\nabla \varphi|^2 \ds, \quad \dom a = H_0^1(\Omega).
\end{equation}

The goal of this paper is to study the spectral problem of $-\Delta_\Omega^D$. 
Inspired by \cite{david}, we assume that
\begin{equation}\label{condderivative}
\lim_{|x| \to \infty} f'(x) =: \beta_1, \quad \lim_{|x| \to \infty} g'(x) =: \beta_2, \quad
\beta_1, \beta_2 \in \mathbb R.
\end{equation}
We find the essential spectrum of $-\Delta_\Omega^D$ and we discuss conditions that ensure the existence of discrete spectrum for this operator. 

Denote by $y:=(y_1, y_2)$ a point of $S$, $\partial_{y_1} := \partial/\partial y_1$,
$\partial_{y_2} := \partial/\partial y_2$. Consider the  two-dimensional operator
\begin{equation}\label{fourpp0}
H_{\beta_1, \beta_2}(0):=-(\beta_1 \partial_{y_1} + \beta_2 \partial_{y_2})^2 
- \partial_{y_1}^2 - \partial_{y_2}^2,
\end{equation}
$\dom H_{\beta_1, \beta_2} (0) = H^2(S) \cap H_0^1(S)$.
Denote by $E_1(0)$ the first eigenvalue of $H_{\beta_1, \beta_2} (0)$ and by $v_1$ the corresponding eigenfunction.
Since $H_{\beta_1, \beta_2}(0)$ is an elliptic operator with real coefficients, $E_1(0)$ is simple and $v_1$
can be chosen to be real and positive in $S$; see, e.g., Chapter 6 of \cite{evans}.
Our first result states
 
\begin{Theorem}\label{sigmadiscintr}
Suppose that the conditions in (\ref{condderivative}) hold. Then,
\[\sigma_{ess} (-\Delta_\Omega^D) = \left[E_1(0), \infty\right).\]
\end{Theorem}

The proof of Theorem \ref{sigmadiscintr} is presented in Section \ref{sectionessspec}. 

The next step is to analyze 
the existence of discrete eigenvalues for $-\Delta_\Omega^D$. Define the constants
\[ A: = \int_S \left| \frac{\partial v_1}{\partial y_1} \right|^2 \dy, \quad
B:= \int_S \frac{\partial v_1}{\partial y_1} \frac{\partial v_1}{\partial y_2} \dy, \quad
C: = \int_S \left| \frac{\partial v_2}{\partial y_1} \right|^2 \dy,\]
and the function
\[V(x) := A (f'(x)^2 - \beta_1^2) + 2 B f'(x) g'(x) + C (g'(x)^2-\beta_2^2), \quad x \in \mathbb R.\]

\begin{Theorem}\label{theomainimpdis}
Suppose that the conditions in (\ref{condderivative}) hold and $V(x) \in L^1(\mathbb R)$.
If $\int_\mathbb R V(x) \dx < 0$, then 
\[\inf \sigma(-\Delta_\Omega^D) < E_1(0),\] 
i.e., $\sigma_{dis} (-\Delta_\Omega^D) \neq \emptyset$.
\end{Theorem}

The proof of this result is presented in Section \ref{sectiondisspec}.
The condition $\int_\mathbb R V(x) \dx < 0$ implies the existence of discrete eigenvalues for
$-\Delta_\Omega^D$. However, it is not a necessary condition for this to happen. 
For example, 

\begin{Theorem}\label{theomaifatend}
Suppose that the conditions in (\ref{condderivative}) hold with $g'=0$.
If $f'(x)^2-\beta_1^2 \in L^1(\mathbb R)$, $f''(x) \in L_{loc}^1(\mathbb R)$,
$\int_\mathbb R (f'(x)^2-\beta_1^2) \dx = 0$ and $f'$ is not constant, then 
\[\inf \sigma(-\Delta_\Omega^D) < E_1(0),\] 
i.e., $\sigma_{dis} (-\Delta_\Omega^D) \neq \emptyset$.
\end{Theorem}

The proof of this result is also presented in Section \ref{sectiondisspec}.
Note that the condition $g'=0$ implies that the reference curve (\ref{refcurveint}) belongs to a plane.
A similar result can be found if $f'=0$, $g'(x)^2-\beta_2^2 \in L^1(\mathbb R)$, $g''(x) \in L_{loc}^1(\mathbb R)$,
$\int_\mathbb R (g'(x)^2-\beta_2^2) \dx = 0$ and $g'$ is not constant.

Now, we are going to show that it is possible to find additional information
about $\sigma_{dis}(-\Delta_\Omega^D)$ provided that $\Omega$ is thin enough. 
For that, we add a small parameter multiplying the points of the cross-section $S$. More precisely,
for $\varepsilon > 0$ small enough, we consider the 
mapping 
\begin{equation}\label{cahngeofcoordif}
\begin{array}{llll}
{\cal L_\varepsilon} : &   \mathbb R \times S    & \to      & \mathbb R^3 \\
                       & (x, y_1, y_2)  & \mapsto  & r(x) + \varepsilon y_1 e_2 + \varepsilon y_2 e_3
\end{array},
\end{equation}
and we define the thin waveguide
\[\Omega_\varepsilon := {\cal L_\varepsilon}(\mathbb R \times S).\] 
Let $-\Delta_{\Omega_\varepsilon}^D$ be the Dirichlet Laplacian operator
in $\Omega_\varepsilon$, i.e., the self-adjoint operator associated with 
\begin{equation}\label{quadinform}
a_\varepsilon(\varphi) : = 
\int_{\Omega_\varepsilon}  |\nabla \varphi|^2 \ds, \quad \dom a_\varepsilon = H_0^1(\Omega_\varepsilon).
\end{equation}
For simplicity, we denote $-\Delta_\varepsilon^D: = -\Delta_{\Omega_\varepsilon}^D$. 
In this case, we have

\begin{Theorem}\label{theothinnumb}
Suppose that the conditions in (\ref{condderivative}) hold and $V(x) \in L^1(\mathbb R)$.
If $V(x)$ assumes a negative value, then
\[\sigma_{dis}(-\Delta_\varepsilon^D) \neq \emptyset,\]
for all $\varepsilon > 0$  small enough. Furthermore,
given $n \in \mathbb N$, there exists $\varepsilon_n > 0$ so that 
the spectrum of $-\Delta_{\varepsilon_n}^D$ contains
at least $n$ discrete eigenvalues, counting multiplicity.
\end{Theorem}

In particular, Theorem  \ref{theothinnumb} ensures that the number
of discrete eigenvalues can be arbitrarily large since the waveguide is thin enough. Its proof 
is presented in Section \ref{sectiondisspec}. In this introduction, we give an alternative proof for the
first statement of the theorem.

At first,  we remember some spectral properties of a self-adjoit operator.
Let $T$ be a self-adjoint operator that is bounded from below and denote by $t(\psi)$ its associated quadratic form.
By Min-Max Principle,
\begin{equation}\label{minmaxpint}
\lambda_1(T): = \inf\{t(\psi)/\|\psi\|^2: 0 \neq \psi \in \dom t\}
\end{equation}
is either a discrete eigenvalue 
or the bottom of the essential spectrum of $T$.
In the conditions of Theorem \ref{theothinnumb}, one has  $\sigma_{ess}(-\Delta_\varepsilon^D) = [E_1(0)/\varepsilon^2, \infty)$.
Thus,
the strategy is to study the quantity $\lambda_1(-\Delta_\var^D)$ in order to
discuss the existence of discrete eigenvalues for $-\Delta_\varepsilon^D$ provided that
$\varepsilon > 0$ is small enough. This will be done using some estimates for the quadratic form $a_\varepsilon(\varphi)$.

Let $\Lambda:= \mathbb R \times S$.
According to the change of coordinates described in Section \ref{sectionchange}, $a_\varepsilon(\varphi)$ 
can be identified with
\begin{equation}\label{bvarbinm}
b_\varepsilon(\psi) := 
\int_\Lambda \left( \left|\psi' - \frac{f'(x)}{\varepsilon} \frac{\partial \psi}{\partial y_1} 
- \frac{g'(x)}{\varepsilon} \frac{\partial \psi}{\partial y_2} \right|^2 
+ \frac{|\nabla_y \psi|^2}{\varepsilon^2}\right) \dx \dy,
\end{equation}
$\dom b_\varepsilon = H_0^1(\Lambda)$.
Recall $y=(y_1,y_2)$ that denotes a point of $S$. Furthermore, 
$\psi':= \partial \psi / \partial x$,   and $\nabla_y \psi:= (\partial_{y_1} \psi, \partial_{y_2} \psi)$.

Consider the closed subspace ${\cal W} := \{ w v_1 : w \in  L^2(\mathbb R)\}$ of the Hilbert space $L^2(\Lambda)$.
Define the one-dimensional quadratic form
\begin{align*}
s_\varepsilon(w) & : =  b_\varepsilon (w v_1) - \frac{E_1(0)}{\varepsilon^2} \|wv_1\|^2\\
& =  
\int_\mathbb R \left( |w'|^2 + \left(A \frac{(f'(x)^2-\beta_1^2)}{\varepsilon^2} |w|^2 + 2 B \frac{f'(x)g'(x)}{\varepsilon^2} 
+ C \frac{(g'(x)^2-\beta_2^2)}{\varepsilon^2}\right) |w|^2 \right) \dx,\\
& = 
\int_\mathbb R \left( |w'|^2 + \frac{V(x)}{\varepsilon^2} |w|^2 \right) \dx,
\end{align*}
$\dom s_\varepsilon:=H^1(\mathbb R)$.
Note that $s_\varepsilon(w)$ is the quadratic form $b_\varepsilon(\psi) - (E_1(0)/\varepsilon^2)\|\psi\|^2$ restricted 
to the subspace
$H_0^1(\Lambda) \cap {\cal W}$ which is identified with $H^1(\mathbb R)$.
The self-adjoint operator associated with $s_\varepsilon(w)$ is 
\[w \mapsto \left(-\Delta_\mathbb R + \frac{V(x)}{\varepsilon^2} \Id\right)w, \quad w \in H^2(\mathbb R),\]
where
$-\Delta_\mathbb R$ denotes the one-dimensional Laplacian operator on $\mathbb R$ and $\Id$ is the identity operator on $L^2(\mathbb R)$.

The next step is to compare the values $\lambda_1(-\Delta_\varepsilon^D)$ and
$\lambda_1(-\Delta_\mathbb R + (V(x)/\varepsilon^2) \Id)$. 
Since ${\cal W} \subset H_0^1(\Lambda)$, by (\ref{minmaxpint}), one has 
\begin{equation}\label{appxeq01}
\lambda_1(-\Delta_\varepsilon^D) - \frac{E_1(0)}{\varepsilon^2} \leq 
\lambda_1\left(-\Delta_\mathbb R + \frac{V(x)}{\varepsilon^2} \Id\right).
\end{equation}
However, by Theorem \ref{theoappx} of Appendix \ref{appendix01}, 
\begin{equation}\label{appxeq02}
\lambda_1\left(-\Delta_\mathbb R + \frac{V(x)}{\varepsilon^2} \Id \right) =  \frac{1}{\varepsilon^2}\inf_{x \in \mathbb R}\{V(x)\} 
+ o(\varepsilon^{-2}),
\end{equation}
for all $\varepsilon > 0$ small enough. 
As a consequence of (\ref{appxeq01}) and (\ref{appxeq02}), if $V(x)$ assumes a negative value, then
$\lambda_1(-\Delta_\varepsilon^D) < E_1(0)/\varepsilon^2$, i.e, $\sigma_{dis}(-\Delta_\varepsilon^D) \neq \emptyset$,
for all $\varepsilon > 0$ small enough.

This work is organized as follows.
In Section \ref{sectionchange} we describe the usual changes of coordinates to work in the Hilbert space
$L^2(\Lambda)$ with the usual metric. 
In Section \ref{sectionessspec} we study the essential spectrum of $-\Delta_\Omega^D$ as well as
the proof of Theorem \ref{sigmadiscintr}. Section \ref{sectiondisspec} is dedicated to study 
the discrete spectrum of $-\Delta_\Omega^D$. In that section, we present the proofs of Theorems
\ref {theomainimpdis}, \ref{theomaifatend} and \ref{theothinnumb}.
In Appendices \ref{appendix001} and \ref{appendix01} we present useful results used in this work.

\section{Change of coordinates}\label{sectionchange}

Recall the mapping ${\cal L}$ and the  quadratic form $a(\varphi)$ given 
by (\ref{cahngeofcoordif1}) and  (\ref{quadinform1}), respectively, in the Introduction.
Several arguments to prove the results of this work are 
based on the study of the quadratic form $a(\varphi)$ which acts in $L^2(\Omega)$; 
again, recall that $\Omega = {\cal L}(\Lambda)$ where $\Lambda = \mathbb R \times S$. 
Then, in this section we perform a change of coordinates such that $a(\varphi)$ starts
to act in the Hilbert space $L^2(\Lambda)$ instead of $L^2(\Omega)$.

At first, note that  $\Omega$ can be identified with the Riemannian manifold
$(\Lambda, G)$, where $G=(G_{ij})$ is the metric induced by ${\cal L}$, i.e.,
\[G_{ij} = \langle {\cal G}_i, {\cal G}_j \rangle = G_{ji}, \quad i,j=1,2,\]
where
\[{\cal G}_1 = \frac{\partial {\cal L}}{\partial x}, \quad 
{\cal G}_2 = \frac{\partial {\cal L}}{\partial y_1}, \quad {\cal G}_3 = \frac{\partial {\cal L}}{\partial y_2} .\]
More precisely,
\[G = \nabla {\cal L} \cdot (\nabla {\cal L})^t = \left(
\begin{array}{ccc}
1 +  (f'(x))^2 + (g'(x))^2 &  f'(x) &  g'(x) \\
f'(x)                                   & 1       &  0 \\
g'(x)                                   &    0                & 1
\end{array} \right), \quad \det G = 1.\]
Since $\Omega$ is homeomorphic to the straight waveguide $\Lambda$, 
we consider the unitary operator
\begin{equation}\label{unituvar}
\begin{array}{llll}
{\cal U}: &   L^2(\Omega)  &  \to &  L^2(\Lambda) \\
        &    \psi   &  \mapsto  &        \psi \circ {\cal L}
\end{array},
\end{equation}
and, we define
\begin{align}\label{compaquadravar}
b_{f',g'}(\psi) & :=  a({\cal U}^{-1} \psi)  
= \int_\Lambda \langle \nabla \psi, G^{-1} \nabla \psi \rangle \sqrt{{\rm det}\, G} \, \dx \dy  \nonumber \\ 
& =  \int_\Lambda \left(\left|\psi' - f'(x) \frac{\partial \psi}{\partial y_1} - 
g'(x) \frac{\partial \psi}{\partial y_2} \right|^2 +  |\nabla_y \psi|^2 \right) \dx \dy,
\end{align}
$\dom b_{f',g'} := {\cal U}(H_0^1(\Omega)) = H_0^1(\Lambda)$; in fact, $f',g' \in L^\infty(\mathbb R)$.
Recall  that $y=(y_1, y_2)$ denotes a point of $S$,
$\psi' = \partial \psi / \partial x$,  and $\nabla_y \psi = (\partial_{y_1} \psi, \partial_{y_2} \psi)$.
Denote by $H_{f',g'}$ the self-adjoint operator associated with the quadratic form
$b_{f',g'}(\psi)$. 

\begin{Remark}{\rm 
We can perform a similar change of coordinates for the quadratic form $a_\varepsilon(\varphi)$ given 
(\ref{quadinform}) in the
Introduction. Recall the mapping ${\cal L}_\varepsilon$ and the quadratic form $b_\varepsilon(\psi)$ given by 
(\ref{cahngeofcoordif})
and (\ref{bvarbinm}), respectively.
In this case, consider the
unitary transformation ${\cal U}_\varepsilon :  L^2(\Omega_\varepsilon)  \to   L^2(\Lambda)$,  
${\cal U}_\varepsilon \psi := \varepsilon^2  \psi \circ {\cal L}_\varepsilon$. Then, some calculations show that
$b_\varepsilon(\psi) =  a_\varepsilon({\cal U}_\varepsilon^{-1} \psi)$,   
$\dom b_\varepsilon = {\cal U}_\varepsilon(H_0^1(\Omega_\varepsilon)) =H_0^1(\Lambda)$. Furthermore,
if we compare
this quadratic form with that in (\ref{compaquadravar}), we can note that $b_\varepsilon = b_{f'/\varepsilon, g'/\varepsilon}$.}
\end{Remark}

\section{Essential spectrum}\label{sectionessspec}

This section is dedicated to the proof of Theorem \ref{sigmadiscintr}.
We start with the following result.

\begin{Proposition}\label{essentialspectigual}
Suppose that the conditions in (\ref{condderivative}) hold.
Then, $\sigma_{ess}(H_{f',g'})  = \sigma_{ess}(H_{\beta_1, \beta_2})$.
\end{Proposition}

The proof of this proposition is discussed in Appendix \ref{appendix001}. 
As a consequence, we are going to study the essential spectrum of the operator
$H_{\beta_1, \beta_2}$ instead of $H_{f',g'}$.

Let ${\cal F}_x: L^2(\Lambda) \to L^2(\Lambda)$ be the partial Fourier transform in the longitudinal variable $x$.
${\cal F}_x$ is a unitary operator and, for functions $\psi \in L^1(\Lambda)$, the explicit expression for this transformation is given by
\[({\cal F}_x \psi)(p,y) = \frac{1}{\sqrt{2 \pi}} \int_\mathbb R e^{-ipx} \psi(x,y) \dx.\]
We consider the operator $\hat{H}_{\beta_1, \beta_2} := {\cal F}_x H_{\beta_1, \beta_2} {\cal F}_x^{-1}$ 
which admits a direct integral decomposition
\[\hat{H}_{\beta_1, \beta_2} = \int_\mathbb R^\oplus H_{\beta_1, \beta_2}(p) \,{\rm d}p,\]
where, for each $p \in \mathbb R$, $H_{\beta_1, \beta_2}(p)$ is the self-adjoint operator associated with
the quadratic form
\[h_{\beta_1, \beta_2}(p)(\psi) = \int_S \left(|ip \psi - \beta_1 \partial_{y_1} \psi - \beta_2 \partial_{y_2} \psi|^2 +
|\nabla_y \psi|^2 \right) \dy, \quad \dom h_{\beta_1, \beta_2}(p) = H_0^1(S).\]
Since $S$ is a bounded domain with $C^2$-boundary, one has 
\[H_{\beta_1, \beta_2}(p) = -(i p  - \beta_1 \partial_{y_1} - \beta_2 \partial_{y_2})^2
-\partial_{y_1}^2 - \partial_{y_2}^2, \quad
\dom H_{\beta_1, \beta_2}(p)= H^2(S) \cap H_0^1(S);\]
if $p=0$, we obtain the operator given by (\ref{fourpp0}) in the Introduction. 

By the compactness of the embedding $H_0^1(S) \hookrightarrow L^2(S)$, each
$H_{\beta_1, \beta_2}(p)$ has purely discrete spectrum. 
Denote by $\{E_n(p)\}_{n \in \mathbb N}$ the
sequence of the eigenvalues of $H_{\beta_1, \beta_2}(p)$ 
and by
$\{v_n(p)\}_{n \in \mathbb N}$ the sequence of the corresponding normalized  eigenfunctions, i.e.,
\[H_{\beta_1, \beta_2}(p) v_n(p) = E_n(p) v_n(p), \quad n \in \mathbb N, \quad p \in \mathbb R.\]
As already done in the Introduction, for simplicity, we denote $v_1 :=v_1(0)$.

Finally,
\begin{equation}\label{eqdecespavsdn}
\sigma(H_{\beta_1, \beta_2}) = \cup_{p \in \mathbb R} \sigma(H_{\beta_1, \beta_2}(p)) =
\cup_{n \in \mathbb N} \{E_n(p): p \in \mathbb R\}.
\end{equation}

\begin{Lemma}\label{lemmanalytic}
Each $E_n(\cdot)$, $n \in \mathbb N$, is a real-analytic function of $p$ and 
\begin{equation}\label{lemmanalyeig}
\displaystyle \lim_{p \to \pm \infty} E_n(p) = \infty.
\end{equation}
\end{Lemma}
\begin{proof}
For each $p \in \mathbb R$, the domain $\dom H_{\beta_1, \beta_2}(p)$ coincides with 
$\dom H_{\beta_1, \beta_2}(0)$, and we write 
\[H_{\beta_1, \beta_2}(p) = H_{\beta_1, \beta_2}(0) +  p^2 + 2ip(\beta_1 \partial_{y_1} + \beta_2 \partial_{y_2}).\]
Now, take $z \in \mathbb C$ with ${\rm img}\,z \neq 0$, and denote $R_z := (H_{\beta_1, \beta_2}(0) - z \Id)^{-1}$.
We have the estimate
\begin{align*}
\|2ip(\beta_1 \partial_{y_1} \psi + \beta_2 \partial_{y_2} \psi)\|^2
&  \leq
4p^2 \, \langle  \psi, H_{\beta_1, \beta_2}(0) \psi \rangle  \\
& \leq 
4 p^2 \, \langle R_z(H_{\beta_1, \beta_2}(0) - z \Id) \psi, H_{\beta_1, \beta_2}(0) \psi \rangle  \\
&\leq 
4 p^2 \, \langle R_z H_{\beta_1, \beta_2}(0) \psi, H_{\beta_1, \beta_2}(0) \psi \rangle  
+  |z| \langle \psi, R_{\overline{z} } H_{\beta_1, \beta_2}(0) \psi \rangle  \\
& \leq 
4 p^2 \, \|R_z H_{\beta_1, \beta_2}(0) \psi\| \|H_{\beta_1, \beta_2}(0) \psi\| +
|z| \langle \psi, (\Id+\overline{z} R_{z}) \psi \rangle \\
& \leq 
4 p^2 \, \left[ \|R_z\|  \|H_{\beta_1, \beta_2}(0) \psi\| + 
  \left( |z| +|z|^2 \|R_{z} \| \right) 
\|\psi\|^2 \right],
\end{align*}
for all  $\psi \in \dom H_{\beta_1, \beta_2}(0)$ and all $p \in \mathbb R$.
Since $\|R_z\| \to 0$, as ${\rm img}\,z \to \infty$, the operator
$2ip(\beta_1 \partial_{y_1} + \beta_2 \partial_{y_2})$ is $H_{\beta_1, \beta_2}(0)$-bounded with
zero relative bound. Consequently, since $p^2$ is clearly analytic, $\{H_{\beta_1, \beta_2}(p): p \in \mathbb R\}$ is a type A analytic family. 
As a consequence, Theorem 3.9 of \cite{kato} implies that all the $E_n(\cdot)$ are real-analytic functions of $p$.

Now, fix $\delta > 0$. For each $\psi \in \dom h_{\beta_1, \beta_2}(p)$, one has
\[(1+\delta) |\beta_1 \partial_{y_1} \psi + \beta_2 \partial_{y_2} \psi|^2 + \frac{p^2}{1+\delta} |\psi|^2
\geq  2\,{\rm Re}\left(ip \psi (\beta_1 \partial_{y_1} \overline{\psi} + \beta_2 \partial_{y_2} \overline{\psi})\right).\]
Consequently,
\begin{align*}
h_{\beta_1, \beta_2}(p)(\psi) 
& =
\int_S \left(p^2|\psi|^2 -
2 \,{\rm Re} (ip \psi 
(\beta_1 \partial_{y_1} \overline{\psi} + \beta_2 \partial_{y_2} \overline{\psi})) + 
|\beta_1 \partial_{y_1} \psi + \beta_2 \partial_{y_2} \psi|^2 \right) \dy \\
& \geq
\int_S \left(p^2 \frac{\delta}{1+\delta}|\psi|^2 - \delta  
|\beta_1 \partial_{y_1} \psi + \beta_2 \partial_{y_2} \psi|^2 \right) \dy.
\end{align*}
Thus, we obtain the pointwise limit $h_{\beta_1, \beta_2}(p)(\psi) \to \infty$, as $p\to \pm \infty$.
In particular, for each $n \in \mathbb N$, $h_{\beta_1, \beta_2}(p)(v_n(p)) = E_n(p) \to \infty$, as $p\to \pm \infty$. Then,
(\ref{lemmanalyeig}) is proven.
\end{proof}

\begin{Remark}{\rm
Lemma \ref{lemmanalytic} ensures that the functions $E_n(p)$ are 
nonconstant and analytic in $\mathbb R$.
Consequently, the spectrum of the operator $H_{\beta_1, \beta_2}$ is purely absolutely continuous;
see Theorem XIII.86 of \cite{reed}.}
\end{Remark}


\begin{Proposition}\label{lemmaessspecve12}
One has $\sigma(H_{\beta_1, \beta_2})=  [E_1(0), \infty)$.
\end{Proposition}
\begin{proof}
Due to decomposition in (\ref{eqdecespavsdn}) and 
Lemma \ref{lemmanalytic}, we have $[E_1(0), \infty) \subset \sigma(H_{\beta_1, \beta_2})$. It remains to show that
$(-\infty, E_1(0)) \cap \sigma(H_{\beta_1, \beta_2}) = \emptyset$. Since $C_0^\infty(S)$ is a core of $h_{\beta_1, \beta_2}(p)$, by Min-Max Principle,
\[\inf \sigma (H_{\beta_1, \beta_2}(p)) = \inf_{0\neq \psi \in C_0^\infty(S)} \left\{
\int_S \left(|ip \psi - \beta_1 \partial_{y_1} \psi - \beta_2 \partial_{y_2} \psi|^2 +
|\nabla_y \psi|^2 \right) \dy / \int_S |\psi|^2 \dy \right\}.\]
Now, we perform the change of variable $\psi = \varphi v_1$, $\varphi \in C_0^\infty(S)$;
note that $v_1^{-1} \psi \in C_0^\infty(S)$ if $\psi \in C_0^\infty(S)$.
Some straightforward computations show that
\begin{align*}
& 
\int_S \left(|ip \psi - \beta_1 \partial_{y_1} \psi - \beta_2 \partial_{y_2} \psi|^2 +
|\nabla_y \psi|^2 \right) \dy \\
& = 
\int_S \left(|ip \varphi - \beta_1 \partial_{y_1} \varphi - \beta_2 \partial_{y_2} \varphi|^2 +
|\nabla_y \varphi|^2 \right) |v_1|^2 \dy \\
& +
\int \left( -(1+\beta_1^2) \frac{\partial^2 v_1}{\partial y_1^2} -(1+\beta_2^2) \frac{\partial^2 v_1}{\partial y_2^2} 
- 2 \beta_1 \beta_2 \frac{\partial^2 v_1}{\partial y_1y_2} \right) v_1 |\varphi|^2 \dy\\
& = 
\int_S \left(|ip \varphi - \beta_1 \partial_{y_1} \varphi - \beta_2 \partial_{y_2} \varphi|^2 +
|\nabla_y \varphi|^2 \right) |v_1|^2 \dy \\
& +
E_1(0) \int_S |\psi|^2 \dy.
\end{align*}
Then, $E_1(0) \leq \inf \sigma(H_{\beta_1, \beta_2}(p))$, for all $p \in \mathbb R$. Thus,
$(-\infty, E_1(0)) \cap \sigma(H_{\beta_1, \beta_2}) = \emptyset$.
\end{proof}

\vspace{0.3cm}
\noindent
{\bf Proof of Theorem \ref{sigmadiscintr}:}
It just to apply 
Propositions \ref{essentialspectigual} and  \ref{lemmaessspecve12}.

\begin{Remark}
{\rm If $\beta_1\beta_2=0$, we can give an alternative proof to find $\sigma_{ess}(-\Delta_\Omega^D)$. In fact, suppose $\beta_2 =0$. Define
\[\gamma:= \frac{p \beta_1}{(1+\beta_1^2)},\]
and consider the multiplication operators $e^{i \gamma y_1}$ and $e^{-i \gamma y_1}$.
Then, for each $p \in \mathbb R$,
\begin{eqnarray*}
-(i p  - \beta_1 \partial_{y_1} )^2  - \partial_{y_1}^2 - \partial_{y_2}^2
& = &
-(1+ \beta_1^2) \left( \partial_{y_1} - \frac{i p \beta_1}{1+\beta_1^2} \right)^2 
- \partial_{y_2}^2  + 
\frac{p^2}{1+\beta_1^2} \\
& = &
e^{i \gamma y_1} \left[ -(1+\beta_1^2) \partial_{y_1}^2 - \partial_{y_2}^2 +
\frac{p^2}{1+\beta_1^2} \right] e^{-i \gamma y_1}.
\end{eqnarray*}
Consequently,
\[\sigma(H_{\beta_1, 0}(p)) = \left\{ E_n(0) + \frac{p^2}{1+\beta_1^2}\right\}_{n=1}^\infty, \quad \hbox{and} \quad
\sigma(H_{\beta_1, 0}) = [E_1(0), \infty).\]}
\end{Remark}

\section{Discrete spectrum}\label{sectiondisspec}

This section is dedicated to prove Theorens \ref{theomainimpdis}, \ref{theomaifatend}
and \ref{theothinnumb} stated in the Introduction. For simplicity, 
we denote $b:=b_{f',g'}$.

\vspace{0.3cm}
\noindent
{\bf Proof of Theorem \ref{theomainimpdis}:}
For each  $n \in \mathbb N$, consider a linear function $\varphi_n : \mathbb R \to \mathbb R$ in $C^\infty(\mathbb R)$ satisfying the following conditions:
\[|| \varphi_n \|_\infty \leq 1, \quad \varphi_n \equiv 1 \,\, \hbox{on} \,\, [-n,n],  \quad
\varphi_n \equiv 0 \,\, \hbox{on} \,\, \mathbb R \backslash (-2n,2n).\]
Define
\[\psi_n(x,y) := \varphi_n(x) v_1(y), \quad n \in \mathbb N,\]
and note that $\psi_n \in \dom b$, for all $n \in \mathbb N$. Some calculations show that
\begin{align*}
b(\psi_n) - E_1(0) \|\psi_n\|^2 & = 
\int_\mathbb R |\varphi_n'|^2 \dx +  \int_\mathbb R 
\left( A (f'(x)^2 - \beta_1^2) + 2 B f'(x) g'(x) + C (g'(x)^2 -\beta_2^2) \right) |\varphi_n|^2 \dx \\
& = 
\int_\mathbb R |\varphi_n'|^2 \dx +  \int_\mathbb R V(x) |\varphi_n|^2 \dx.
\end{align*}
Since $\|\varphi'_n\|^2 = 2/n$, for all $n \in \mathbb N$, by dominated convergence theorem,
\[b(\psi_n) - E_1(0) \|\psi_n\|^2 \to  
\int_\mathbb R V(x) \dx, \quad n \to \infty.\]
Thus, there exists $n$ such that $b(\psi_n) - E_1(0) \|\psi_n\|^2 < 0$.

\vspace{0.3cm}
\noindent
{\bf Proof of Theorem \ref{theomaifatend}:}
For each $n \in \mathbb N$, consider $\varphi_n$ and $\psi_n$ as in the proof of Theorem \ref{theomainimpdis}. 
However, we add a small perturbation and we define
\[\psi_{n, \delta}(x,y) := \psi_n(x,y) + \delta \xi(x) y_1 v_1(y), \quad n \in \mathbb N,\]
where $\delta$ is a real number and $\xi \in C_0^\infty(\mathbb R)$. In this case, we have
\begin{align*}
&  
\lim_{n \to \infty} \left(b(\psi_{n, \delta}) - E_1(0) \| \psi_{n, \delta} \|^2 \right) \\
& = 
\lim_{n \to \infty}  
\left[b(\psi_n) - E_1(0) \|\psi_n\|^2 + 2 \delta 
\left(b(\psi_n, \xi y_1 v_1)-E_1(0) \int_\Lambda \psi_n \xi y_1 v_1 \dx\dy \right) \right. \\
& + 
\left. \delta^2 \left( b(\xi y_1 v_1) - E_1(0) \|\xi y_1  v_1\|^2 \right) \right] \\
&  =  
\lim_{n \to \infty} \left[ 2 \delta 
\left(b(\psi_n, \xi y_1 v_1)-E_1(0) \int_\Lambda \psi_n \xi y_1 v_1 \dx \dy\right) + 
\delta^2 \left( b(\xi y_1 v_1) - E_1(0) \|\xi y_1  v_1\|^2 \right) \right].
\end{align*}
In the last equality was used the assumption $\int_\mathbb R (f'(x)^2-\beta_1^2) \dx = 0$.
Now, we need to show that there exists a function $\xi$ satisfying 
\begin{equation}\label{maisineqproof}
\lim_{n \to \infty} \left(b(\psi_n, \xi y_1 v_1) - E_1(0) \| \psi_n, \xi y_1 v_1\|^2 \right) \neq 0.
\end{equation}
In fact, if (\ref{maisineqproof}) holds true, it is enough 
to choose  $\delta$ such that
$\left(b(\psi_{n, \delta}) - E_1(0) \| \psi_{n, \delta} \|^2 \right) < 0$, for some $n$ large enough.

Define the constant
\[ \tilde{A} := \int_S y_1 \left(\frac{\partial v_1}{\partial y_1}\right)^2 \dy. \]
Some calculations show that

\begin{align*}
& 
\lim_{n \to \infty} \left[ b(\psi_n, \xi y_1 v_1) - E_1(0) \int_\Lambda \psi_n  \xi y_1  v_1 \, \dx \dy \right]    \\ 
& = 
\lim_{n \to \infty}
\int_\Lambda \left(\varphi_n' v_1 - f'(x) \varphi_n \frac{\partial v_1}{\partial y_1} \right) 
\left(\xi' y_1 v_1 - f'(x) \xi \, v_1 - f'(x) \xi  y_1 \frac{\partial v_1}{\partial y_1} \right) \dx \dy \\
& + 
\lim_{n \to \infty} \int_\Lambda  \varphi_n \xi \langle \nabla_y v_1, \nabla_y (y_1 v_1) \rangle \dx \dy
-
E_1(0)  \lim_{n \to \infty} \int_\Lambda \varphi_n \xi y_1 v_1^2 \, \dx\dy \\
& = 
\int_\Lambda \left(- f'(x) \frac{\partial v_1}{\partial y_1}  \xi' y_1 v_1 
+ (f'(x))^2  \frac{\partial v_1}{\partial y_1} \xi v_1
+ (f'(x))^2 \xi y_1 \left(\frac{\partial v_1}{\partial y_1}\right)^2 \right) \dx \dy \\
& - 
\int_\Lambda \xi y_1 |\nabla_y v_1|^2 \dx \dy - E_1(0) \int_\Lambda   \xi y_1 v_1^2 \, \dx \dy \\
& = 
\int_\Lambda \left(- f'(x) \frac{\partial v_1}{\partial y_1}  \xi' y_1 v_1 
+ (f'(x))^2 \xi y_1 \left(\frac{\partial v_1}{\partial y_1}\right)^2 \right) \dx \dy \\
& - 
\int_\Lambda \xi y_1 |\nabla_y v_1|^2 \dx \dy - E_1(0) \int_\Lambda   \xi y_1 v_1^2 \, \dx \dy   \\
& = 
\int_\Lambda \xi \left(f''(x) \frac{\partial v_1}{\partial y_1}   y_1 v_1 
+ (f'(x))^2 y_1 \left(\frac{\partial v_1}{\partial y_1}\right)^2 \right) \dx \dy \\
& - 
\int_\Lambda \xi y_1 |\nabla_y v_1|^2 \dx \dy - E_1(0) \int_\Lambda   \xi y_1 v_1^2 \, \dx \dy   \\
& = 
\int_\Lambda \xi \left(f''(x) \frac{\partial v_1}{\partial y_1}   y_1 v_1 
+ (f'(x)^2 - \beta_1^2) y_1 \left(\frac{\partial v_1}{\partial y_1}\right)^2 \right) \dx \dy \\
& = 
\int_\mathbb R \xi \left( - \frac{f''(x)}{2}  + \tilde{A} (f'(x)^2 - \beta_1^2)  \right) \dx.
\end{align*}

If the last integral is zero for all possible choices of $\xi$, then the function $f$ satisfy the differential equation
\begin{equation}\label{eqaedona}
-\tau_1'(x) + 2 \tilde{A} (\tau_1^2(x) + 2 \beta_1 \tau_1(x))=0,
\end{equation}
where $\tau_1(x) = f'(x)-\beta_1$. The family of curves
\begin{equation}\label{solutionadly}
\tau_{1,c}(x) = \frac{2 \beta_1}{c e^{-4\tilde{A} \beta_1 x} -1}, \quad c \in \mathbb R,
\end{equation}
describes all the solutions of the equation (\ref{eqaedona}). 
Now, note that
$c> 0$ (resp. $c< 0$) implies $f'(x)^2-\beta_1^2 > 0$ (resp. $f'(x)^2-\beta_1^2 < 0$);
the case $c> 0$ adimts  singularity as $x= (4\tilde{A}\beta_1)^{-1} \ln c$.
The case $c=0$ implies that $f'$ is constant. Then, the solutions in (\ref{solutionadly}) are not admissible.
Thus, there exists a function $\xi$ satisfying 
(\ref{maisineqproof}).

\begin{Remark}{\rm
The differential equation (\ref{eqaedona}) is very similar to that find in the proof Theorem 1.2 of \cite{david}.
Intuitively, we can justify this similarity by the fact that, in the conditions of Theorem \ref{theomaifatend},
the reference curve (\ref{refcurveint}) belongs to a plane.
}
\end{Remark}

\vspace{0.3cm}
\noindent
{\bf Proof of Theorem \ref{theothinnumb}:}
The case $n=0$ is trivial. Suppose $n \geq 1$. 
We are going to show that there exists $\varepsilon_n > 0$ so that
the spectrum of $-\Delta_{\varepsilon_n}^D$ contains at least $n$ discrete eigenvalues, counting
multiplicity.
Let $I \subset \mathbb R$ be a bounded interval so that
$V(x) < 0$, for all $x \in I$. 
Define $x_0 := \inf I$ and $x_i := x_0 + i |I|/n$, for all $i \in \{1, \dots, n\}$.
Let $\psi_0$ a non-zero function from $H^1(\mathbb R)$
so that ${\rm supp} \, \psi_0 \subset (x_0, x_1)$. For each $i \in \{1,\dots, n\}$, define
\[\Phi : =\int_{x_{i-1}}^{x_i} |\psi_0(x_0 + x - x_{i-1})|^2 \dx , \quad
\psi_i(x) : = \Phi^{-2} \, \psi_0(x_0 + x - x_{i-1}), \quad x \in \mathbb R,\]
and
\[\varphi_i(x,y) : =  \psi_i(x) v_0(y), \quad (x,y) \in \Lambda.\]
The subset $\{ \varphi_i\}_{i=1}^n$ is a basis of a subspace of $H_0^1(\Lambda)$. Furthermore, for $i \neq j$, since
$\varphi_i$ and $\varphi_j$ have disjoint
supports, one has $b_\varepsilon(\varphi_i, \varphi_j) = 0$.
Some calculations show that 
\[b_\varepsilon (\varphi_i, \varphi_i) - \frac{E_1(0)}{\varepsilon^2} \|\varphi_i\|^2 = 
\int_\mathbb R  \left( |\varphi_i'|^2 + \frac{V(x)}{\varepsilon^2} |\varphi_i|^2    \right) \dx < 0,\]
for all $i \in \{1, \cdots, n\}$, for all $\varepsilon > 0$ small enough. Then, by taking $\varepsilon_n > 0$ small enough, the result
follows by Lemma 4.5.4 of \cite{livrodavies} and Theorem \ref{sigmadiscintr}.

\appendix

\section{Appendix}\label{appendix001}

\noindent
{\bf Stability of the essential spectrum}
\vspace{0.3cm}

This appendix is dedicated to the proof of Proposition \ref{essentialspectigual}.
We use the arguments of \cite{david}. In that work, the authors employed a different characterization of the spectrum 
essential which can be adapted to our problem. In fact, 

\begin{Lemma}\label{lemmasimilar}
A real number $\lambda$ belongs to the essential spectrum of $H_{f', g'}$ if, and only if, there exists a sequence 
$\{\psi_n\}_{n=1}^\infty \subset \dom b_{f',g'}$ such that the following conditions hold:

\vspace{0.2cm}
\noindent
(i) $\| \psi_n\| =1$, for all $ n \geq 1$;

\vspace{0.2cm}
\noindent
(ii) $(H_{f', g'}-\lambda \Id)\psi_n \to 0$, as $n \to \infty$, in the norm of the dual space $(\dom b_{f', g'})^*$;

\vspace{0.2cm}
\noindent
(iii) ${\rm supp}\,\psi_n \subset Q \backslash (-n,n) \times S$, for all $n \geq 1$.
\end{Lemma}

The proof of Lemma \ref{lemmasimilar} is very similar to the proof of Lemma 5 of \cite{david}, it will not 
presented in this text.

\vspace{0.3cm}
\noindent
{\bf Proof of Propositon \ref{essentialspectigual}:}
Let $\lambda \in \sigma_{ess}(H_{\beta_1, \beta_2})$. By Lemma \ref{lemmasimilar}, there exists a sequence 
$\{\psi_n\}_{n=1}^\infty \subset \dom b_{\beta_1, \beta_2}$ such that
the conditions $(i)-(iii)$ are satisfied. For simplicity, write
\[\tau_1(x):= f'(x) - \beta_1, \quad \tau_2(x):= g'(x) - \beta_2.\]
Some calculations show that
\begin{align*}
b_{f',g'}(\varphi, \psi_n)  -  \lambda \langle \varphi, \psi_n \rangle  
& = 
b_{\beta_1, \beta_2} (\varphi, \psi_n) -  \lambda \langle \varphi, \psi_n \rangle  \\
& + 
\int_\Lambda \left(-\tau_1 \frac{\partial \varphi}{\partial y_1} -\tau_2 \frac{\partial \varphi}{\partial y_2}\right)
\left(\psi_n' - \beta_1 \frac{\partial \psi_n}{\partial y_1} - \beta_2 \frac{\partial \psi_n}{\partial y_2}\right)  \dx \dy \\
& +  
\int_\Lambda \left(\varphi' - (\beta_1 + \tau_1) \frac{\partial \varphi}{\partial y_1} - 
(\beta_2+\tau_2) \frac{\partial \varphi}{\partial y_2}\right)
\left(-\tau_1 \frac{\partial \psi_n}{\partial y_1} -\tau_2 \frac{\partial \psi_n}{\partial y_2}\right) 
\dx \dy.
\end{align*}
Since $\|\partial \varphi/ \partial y_1\|^2, \| \partial \varphi/ \partial y_2\|^2 \leq b_{f',g'}(\varphi) = \| \varphi \|_+^2$,
one has

\begin{align*}
&  
\sup_{0 \neq \varphi \in H_0^1(\Lambda)} \left\{
\int_\Lambda \left(-\tau_1 \frac{\partial \varphi}{\partial y_1} -\tau_2 \frac{\partial \varphi}{\partial y_2}\right)
\left(\psi_n' - \beta_1 \frac{\partial \psi_n}{\partial y_1} - \beta_2 \frac{\partial \psi_n}{\partial y_2}\right)  \dx \dy /
\| \varphi \|_+ \right\} \\
& \leq 
\left( \|\tau_1\|_{L^\infty(\mathbb R \backslash (-n,n))} + \|\tau_2\|_{L^\infty(\mathbb R \backslash (-n,n))} \right) 
\left( \|\psi_n'\| + \beta_1 \left \| \frac{\partial \psi_n}{\partial y_1} \right\| + \beta_2 \left \| \frac{\partial \psi_n}{\partial y_2} \right \| \right)
\to 0,
\end{align*}
and
\begin{align*}
& 
\sup_{0 \neq \varphi \in H_0^1(\Lambda)} \left\{
\int_\Lambda
\left(\varphi' - (\beta_1 + \tau_1) \frac{\partial \varphi}{\partial y_1} - (\beta_2+\tau_2) \frac{\partial \varphi}{\partial y_2}\right)
\left(-\tau_1 \frac{\partial \psi_n}{\partial y_1} -\tau_2 \frac{\partial \psi_n}{\partial y_2}\right) 
\dx \dy  / \|\varphi\|_+ \right \} \\
& \leq 
\|\tau_1\|_{L^\infty(\mathbb R \backslash (-n,n))} \left \| \frac{\partial \psi_n}{\partial y_1} \right\|  + 
\|\tau_2\|_{L^\infty(\mathbb R \backslash (-n,n))}  \left \| \frac{\partial \psi_n}{\partial y_2} \right \|
\to 0,
\end{align*}
as $n \to \infty$. Then, $\lambda \in \sigma_{ess}(H_{f',g'})$. The inclusion 
$\sigma_{ess}(H_{f',g'}) \subset \sigma_{ess}(H_{\beta_1, \beta_2})$ can be obtained in a similar way.

\section{Appendix}\label{appendix01}

\noindent
{\bf Asymptotic behavior of eigenvalues}

\vspace{0.3cm}

\noindent
Let $W: \mathbb R \to \mathbb R$ be a bounded function and $\mu \in \mathbb R \backslash \{0\}$. 
Define $W_{min} := \, {\rm inf}_{x \in \mathbb R} \{W(x)\}$.
Consider the quadratic form
\[n_\mu(\phi) = \int_\mathbb R |\phi'|^2 \dx + \mu \int_\mathbb R W(x) |\phi|^2 \dx, \quad
\dom n_\mu = H^1(\mathbb R).\]
Denote by $n_\mu(\psi, \varphi)$ and $N_\mu$  its sesquilinear form and its self-adjoint operator associated, respectively.
Consider the sequence $\{\lambda_j(N_\mu)\}_{j \in \mathbb N}$ given by Min-Max Principle; 
see, e.g., Theorem XIII.1 of \cite{reed}.

The next result is a simpler version of Theorem 4 of \cite{freitasdavid}. 

\begin{Theorem}\label{theoappx}
For each $j \in \mathbb N$,
\[\liminf_{\mu \to +\infty} \frac{\lambda_j(N_\mu)}{\mu} = W_{min}.\] 
\end{Theorem}
\begin{proof}
Since $N_\mu \geq \mu W_{min} \Id$, by Min-Max Principle, 
\[\liminf_{\mu \to +\infty} \frac{\lambda_j(N_\mu)}{\mu} \geq  W_{min}.\]
Now, we need to obtain the opposite estimate. 
Consider the multiplication operator
\[{\cal M}_W \phi = W \phi, \quad \dom {\cal M}_W = \{ \phi \in L^2(\mathbb R): W \phi \in L^2(\mathbb R)\}.\]
The spectrum ${\cal M}_W$ is purely essential and equal to the essential range of the generated function $W$.
In particular, $W_{min} \in \sigma({\cal M}_W)$. By the spectral theorem, there exists a sequence
$\{ \psi_i \}_{i \in \mathbb N}$ orthonormalized in $L^2(\mathbb R)$ so that
$\| ({\cal M}_W - W_{min} \Id) \psi_i\| \to 0$, as $i \to \infty$.
Since $C_0^\infty(\mathbb R)$ is dense in $\dom {\cal M}_W$, it follows that there is also a
sequence $\{ \varphi_i \}_{i \in \mathbb N}$ satisfying
\[ \langle \varphi_i, \varphi_j \rangle - \delta_{ij} \to 0, \quad
\langle \varphi_i,  ({\cal M}_W- W_{min}) \varphi_j \rangle \to 0,\]
as $i,j \to \infty$.
Now, given $N \in \mathbb N$, take $k=k(\mathbb N)$ sufficiently large so that
\[A(N) - W_{min} \Id \leq N^{-1} \Id,\]
where $A(N)$ is a symmetric matrix with entries
$\langle \varphi_{i+k}, \varphi_{j+k} \rangle$, for $i, j \in \{1, \dots, N\}$.
Since the subspace generated by $\{\varphi_{1+k}, \dots, \varphi_{N+k}\}$
is a $N$-dimensional subspace of $\dom N_\mu$, one has
$\lambda_j(N_\mu) \leq c_j(N_\mu)$, for all $j \in \{1, \dots, N\}$, 
where $\{c_j(N_\mu)\}_{j=1}^N$ are the eigenvalues
(written in increasing order and repeated according to multiplicity) of
the matrix $C(N_\mu) : = C_{ij}(μN_\mu)$ defined by
$C_{ij}(μN_\mu) = n_\mu(\varphi_{i+k}, \varphi_{i+j})$.
Now, we can see that
\[C(N_{\mu}) \leq \mu (W_{min} + N^{-1}) \Id + d(N) \Id,\]
where $d(N)$ denotes the maximal eigenvalue of the matrix with entries $(\langle \nabla \varphi_{i+k}, \nabla \varphi_{j+k} \rangle)$.
Then, follows that
\[\liminf_{\mu \to +\infty} \frac{\lambda_j(N_\mu)}{\mu} \leq  W_{min} + N^{-1},\]
for $j \in \{1,2, \cdots, N\}$, with $N$ being arbitrarily large.
\end{proof}



\end{document}